\documentclass[]{siamltex}

\usepackage{amsmath,amstext,amssymb,amsfonts}
\usepackage{mathrsfs}
\usepackage{epsfig}
\usepackage{graphicx}

\newtheorem{example}{Example}[section]
\newtheorem{algorithm}{Algorithm}[section]
\renewcommand{\theequation}{\arabic{section}.\arabic{equation}}

\renewcommand{\thefigure}{\arabic{section}.\arabic{figure}}

\newcommand{\pa}{\partial}

\newcommand{\na}{\nabla}

\newcommand{\Om}{\Omega}
\newcommand{\de}{\delta}

\newcommand{\Lam}{\Lambda}
\newcommand{\lam}{\lambda}

\newcommand{\vep}{\varepsilon}
\newcommand{\lj}{|{\hskip -1pt} \|}
\newcommand{\rj}{|{\hskip -1pt} \|}

\newcommand{\diam}{\mathrm{diam}}

\newcommand{\cM}{{\cal M}}

\newcommand{\E}{{\mathbb{E}}}

\newcommand{\R}{{\mathbf{R}}}

\newcommand{\be}{\begin{eqnarray}}
\newcommand{\ee}{\end{eqnarray}}
\newcommand{\beq}{\begin{equation}}
\newcommand{\eeq}{\end{equation}}
\newcommand{\ben}{\begin{eqnarray*}}
\newcommand{\een}{\end{eqnarray*}}
\newcommand{\nn}{\nonumber}

\begin{document}

\makeatletter
\renewcommand\theequation{\thesection.\arabic{equation}}
\@addtoreset{equation}{section}
\makeatother

\title{Stochastic Convergence of A Nonconforming Finite Element Method for the Thin Plate Spline Smoother for Observational Data}
\author{Zhiming Chen\thanks{LSEC, Institute of Computational Mathematics,
Academy of Mathematics and System Sciences and School of Mathematical Science, University of
Chinese Academy of Sciences, Chinese Academy of Sciences,
Beijing 100190, China. This author was supported in part by
the China NSF under the grant
113211061. ({\tt zmchen@lsec.cc.ac.cn}).}
\and Rui Tuo\thanks{Institute of Systems Science,
Academy of Mathematics and System Sciences, Chinese Academy of Sciences,
Beijing 100190, China. ({\tt tuorui@amss.ac.cn}).}
\and Wenlong Zhang \thanks{School of Mathematical Science, University of
Chinese Academy of Sciences, Chinese Academy of Sciences,
Beijing 100190, China. 
({zhangwl@lsec.cc.ac.cn}).}}

\date{}
\maketitle

\begin{abstract}
The thin plate spline smoother is a classical model for finding a smooth function from the knowledge of its observation at scattered locations which may have random noises. We consider a nonconforming Morley finite element method to approximate the model. We prove the stochastic convergence of the finite element method which characterizes the tail property of the probability distribution function of the finite element error. We also propose a self-consistent iterative algorithm to determine the smoothing parameter based on our theoretical analysis. Numerical examples are included to confirm the theoretical analysis and to show the competitive performance of the self-consistent algorithm for finding the smoothing parameter.
\end{abstract}

\noindent {\bf Key words.} Thin plate spline, Morley element, stochastic convergence, optimal parameter choice.

\section{Introduction}
The thin plate spline smoother is a classical mathematical model for finding a smooth function from the knowledge of its observation at scattered locations which may subject to random noises.  Let $\Omega$ be a bounded Lipschitz domain in
$\R^d$ ($d \leq 3$) and $u_0\in H^2(\Om)$ be the unknown smooth function. Let $\{x_i\}^n_{i=1}\subset\Omega$ be the
scattered locations in the domain where the observation is taken. We want to approximate $u_0$ from the noisy data $y_i=u_0(x_i)+e_i,\ 1\leq i \leq n$, where $\{e_i\}^n_{i=1}$ are independent and identically distributed random variables on some probability space ($\mathfrak{X},\mathcal{F},\mathbb{P})$ satisfying
$\mathbb{E}[e_i]=0$ and $\mathbb{E}[e^2_i]\leq \sigma^2$.  Here and in the following $\mathbb{E}[X]$ denotes the expectation of the random variable $X$. The thin plate spline smoother, i.e., $D^2$-spline smoother to approximate $u_0$, is defined to be the unique solution of the following variational problem
\beq\label{p1}
\mathop {\rm min}\limits_{u\in H^2(\Omega)}\frac{1}{n}\sum\limits_{i=1}^{n} {(u(x_i)-y_i)^2+\lambda_n |u|_{H^2(\Omega)}^2},
\eeq
where $\lambda_n>0$ is the smoothing parameter.

The spline model for scattered data has been extensively studied in the literature. For $\Om=\R^d$, \cite{Duchon} proved
that \eqref{p1} has a unique solution $u_n\in H^2(\R^d)$ when the set $\mathbb{T}=\{x_i:i=1,2,\cdots,n\}$ is not collinear (i.e.
the points in $\mathbb{T}$ are not on the same plane).  Explicit formula of the solution is constructed in \cite{Duchon} based on radial basis functions. \cite{Utreras} derived the convergence
rate for the expectation of the error $|u_n-u_0|_{H^j(\Om)}^2$, $j=0,1,2$. Under the assumption that $e_i$, $i=1,2,\cdots,n$, are
also sub-Gaussion random variables, \cite{Geer} proved the stochastic convergence of the error in terms of the empirical norm $\|u_n-u_0\|_n:=(n^{-1}\sum^n_{i=1}|u_n(x_i)-u_0(x_i)|^2)^{1/2}$ when $d=1$. The stochastic convergence which provides additional tail information of the probability distribution function for the random error is very desirable for the approximation of random variables. We refer to \cite{Wahba} for further information of the thin plate spline smoothers.

It is well-known that the numerical method based on radial basis functions to solve the thin plate spline smoother requires to solve
a symmetric indefinite dense linear system of equations of the size $O(n)$, which is challenging for applications with very large data sets \cite{Roberts}. Conforming finite element methods for the solution of thin plate model are studied in \cite{Arcang} and the references therein. In \cite{Roberts} a mixed finite element method for solving $\na u_n$ is proposed and the expectation of the finite element error is proved. The advantage of the mixed finite element method in \cite{Roberts} lies in that one can use
simple $H^1(\Om)$-conforming finite element spaces. The $H^1$ smoother in \cite{Roberts} that the mixed finite element method aims to approximate is not equivalent to the thin plate spline model \eqref{p1}.

In this paper we consider the nonconforming finite element approximation to the problem \eqref{p1}. We use the Morley element
\cite{Morley, Shi, wang} which is of particular interest for solving fourth order PDEs since it has the least number of degrees of freedom on each element. The difficulty of the finite element analysis for the thin plate smoother is the low stochastic regularity of the solution $u_n$. One can only prove the boundedness of $\E[|u_n|_{H^2(\Om)}^2]$ (see Theorem \ref{thm:2.1} below). This difficulty is overcome by a smoothing operator based on the $C^1$-element for
any Morley finite element functions. We also prove the probability distribution function of the empirical norm of the finite element error has an exponentially decaying tail. For that purpose we also prove the convergence of the error $\|u_n-u_0\|_n$ in terms of
the Orlicz $\psi_2$ norm (see Theorem \ref{thm:4.1} below) which improves the result in \cite{Geer}.

One of the central issues in the application of the thin plate model is the choice of the smoothing parameter $\lambda_n$. In the literature it is usually made by the method of cross validation \cite{Wahba}. The analysis in this paper suggests the optimal choice should be
\beq\label{k1}
\lam_n^{1/2+d/8}=O(\sigma n^{-1/2}(|u_0|_{H^2(\Om)}+\sigma n^{-1/2})^{-1}).
\eeq
Since one does not know $u_0$ and the upper bound of the variance $\sigma$ in practical applications, we propose a self-consistent algorithm to determine $\lam_n$ from the natural initial guess $\lam_n=n^{-\frac{4}{4+d}}$. Our numerical experiments show this self-consistent algorithm performs rather well.

The layout of the paper is as follows. In section 2 we recall some preliminary properties of the thin plate model. In section 3 we
introduce the nonconforming finite element method and show the convergence of the finite element solution in terms of the expectation of Sobolev norms. In section 4 we study the tail property of the probability distribution function for the finite element
error based on the theory of empirical process for sub-Gaussion noises. In section 5 we introduce our self-consistent algorithm
for finding the smooth parameter $\lam_n$ and show several numerical examples to support the analysis in this paper.

\section{The thin plate model}

In this section we collect some preliminary results about the thin plate smoother \eqref{p1}. In this paper, we will always assume that $\Omega$ is a bounded Lipschitz domain satisfying the uniform cone condition. We will also assume that $\mathbb{T}$ are uniformly distributed
in the sense that \cite{Utreras} there exists a constant $B>0$ such that $\frac{h_{\max}}{h_{\min}} \leq B$, where
\ben
h_{\max}=\mathop {\rm sup}\limits_{x\in \Omega} \mathop {\rm inf}\limits_{1 \leq i \leq n} |x-x_i| ,\ \ \ \
h_{\min}=\mathop {\rm inf}\limits_{1 \leq i \neq j \leq n} |x_i-x_j|.
\een
It is easy to see that there exist constants $B_1,B_2$ such that $B_1n^{-1/d}\le h_{\max}\le Bh_{\min}\le B_2n^{-1/d}$.

We write the empirical inner product between the data and any function $v\in C(\bar\Omega)$ as $(y,v)_n=\frac{1}{n}\sum^n_{i=1}y_iv(x_i)$. We also write $(u,v)_n=\frac{1}{n}\sum^n_{i=1}u(x_i)v(x_i)$ for any $u,v\in C(\bar\Omega)$ and the empirical norm $\|u\|_n=(\frac{1}{n}\sum_{i=1}^{n} u^2(x_i))^{1/2}$ for any $u\in C(\bar\Om)$. By \cite[Theorems 3.3-3.4]{Utreras}, there exists a constant $C>0$ depending only on $\Om,B$ such that for any $u\in H^2(\Om)$ and sufficiently small $h_{\max}$,
\beq\label{f2}
\|u\|_{L^2(\Om)}\le C(\|u\|_n+h_{\max}^2|u|_{H^2(\Om)}),\ \ \|u\|_n\le C(\|u\|_{L^2(\Om)}+h_{\max}^2|u|_{H^2(\Om)}).
\eeq
It follows from \eqref{f2} and Lax-Milgram lemma that the minimization problem \eqref{p1} has a unique solution $u_n\in H^2(\Om)$. The following convergence result is proved in \cite{Utreras}.

\begin{lemma}\label{lem:2.1}
Let $U_n\in H^2(\R^d)$ be the solution of following variational problem:
\beq\label{p2}
\mathop {\rm min}\limits_{u\in D^{-2}L^2(\R^d)}\|u-y\|^2_n+\lambda_n |u|_{H^2(\R^d)}^2,
\eeq
where $D^{-2}L^2(\R^d)=\{u | D^{\alpha}u \in L^2(\R^d),\ |\alpha|=2 \}$.
Then there exist constants $\lambda_0 > 0$ and $C>0$ such that for any $\lambda_n \leq \lambda_0$ and $n\lambda_n^{d/4}\ge 1$,
\be
& &\E\big[\|U_n-u_0\|^2_n\big] \leq C \lambda_n |u_0|^2_{H^2(\Omega)} + \frac{C\sigma^2}{n\lambda^{d/4}_n},\label{p3}\\
& &\label{p4}
\E \big[|U_n|^2_{H^2(\Om)}\big] \leq C |u_0|^2_{H^2(\Omega)} + \frac{C\sigma^2}{n\lambda^{1+d/4}_n}.
\ee
\end{lemma}

Define the bilinear form $a:H^2(\Om)\times H^2(\Om)\to\R$ as
\beq\label{a2}
a_\Om(u,v) = {\sum\limits_{1\leq i,j\leq d}\int_{\Omega}  {\frac{\partial^2 u}{\partial x_i \partial x_j}\frac{\partial^2 v}{\partial x_i \partial x_j}}}dx,\ \ \ \ \forall u,v\in H^2(\Om).
\eeq
It is obvious that $|u|_{H^2(\Om)}^2=a(u,u)$ for any $u\in H^2(\Om)$. 

\begin{theorem}\label{thm:2.1}
Let $u_n\in H^2(\Om)$ be the unique solution of (\ref{p1}).
Then there exist constants $\lambda_0 > 0$ and $C>0$ such that for any $\lambda_n \leq \lambda_0$ and $n\lambda_n^{d/4}\ge 1$,
\be
& &\label{p5}
\mathbb{E} \big[\|u_n-u_0\|^2_n\big] \leq C \lambda_n |u_0|^2_{H^2(\Omega)} + \frac{C\sigma^2}{n\lambda^{d/4}_n},\\
& &\label{p6}
\mathbb{E} \big[|u_n|^2_{H^2(\Omega)}\big] \leq C |u_0|^2_{H^2(\Omega)} + \frac{C\sigma^2}{n\lambda^{1+d/4}_n}.
\ee
\end{theorem}

\begin{proof} It is clear that $u_n\in H^2(\Om)$ and $U_n\in H^2(\R^d)$ satisfy the following variational forms:
\be
& &\label{p7}
\lambda_n a_\Om(u_n,v)+(u_n,v)_n =(y,v)_n, \ \forall v\in H^2(\Omega),\\
& &\label{p8}
\lambda_n a_{\R^d}(U_n,w)+(U_n,w)_n=(y,w)_n, \ \forall w\in H^2(\R^d).
\ee
Let $F:H^2(\Om)\to D^{-2}L^2(\R^d)$ be the extension operator defined by
\ben
Fu=\mathop{\rm argmin}\limits_{v\in D^{-2}L^2(\R^d),v|_\Om=u}\ |v|_{H^2(\Om)}.
\een
It is known \cite{Duchon, Utreras} that $Fu=u$ in $\Om$ and $|Fu|_{H^2(\R^d)}\le C|u|_{H^2(\Om)}$ for some constant $C>0$. We write $\tilde u=Fu$ in $\R^d$ in the following.
Thus, it follows from \eqref{p7}-\eqref{p8} that
\ben
\lambda_na_{\Om}(u_n-U_n,v)+(u_n-U_n,v)_n=\lambda_n a_{\R^d\backslash\bar\Om}(U_n,\tilde v),\ \ \ \ \forall v\in H^2(\Om),
\een
which implies by taking $v=u_n-U_n|_\Om\in H^2(\Om)$ that
\ben
\lambda_n |u_n-U_n|_{H^2(\Om)}^2+\|u_n-U_n\|_n^2&\le&\lam_n|U_n|_{H^2(\R^d)}|\tilde u_n-\tilde U_n|_{H^2(\R^d)}\\
&\le&C\lam_n|U_n|_{H^2(\R^d)}|u_n-U_n|_{H^2(\Om)},
\een
where $\tilde U_n=F(U_n|_\Om)$.
Therefore
\beq\label{a1}
|u_n-U_n|^2_{H^2(\Om)}\le C|U_n|_{H^2(\R^d)}^2,\ \ \ \ \|u_n-U_n\|_n^2\le\lam_n|U_n|^2_{H^2(\R^d)}.
\eeq
Since $U_n$ is the solution of \eqref{p2} and $\tilde U_n=U_n$ in $\Om$, we have $|U_n|_{H^2(\R^d)}\le |\tilde U_n|_{H^2(\R^d)}\le C|U_n|_{H^2(\Om)}$. Therefore, $\mathbb{E}[|u_n|_{H^2(\Om)}^2]\le C\mathbb{E}[|U_n|_{H^2(\Om)}^2]$, which
implies \eqref{p6} by using \eqref{p4}. Similarly one obtains \eqref{p5} from the second estimate in \eqref{a1} and \eqref{p3}-\eqref{p4}. This completes the proof.
\end{proof}

Theorems \ref{lem:2.1} and \ref{thm:2.1} suggest that an optimal choice of the parameter $\lam_n$ is such that $\lam_n^{1+d/4}=O((\sigma^2n^{-1})|u_0|_{H^2(\Om)}^{-2})$.

\section{Nonconforming finite element method}

In this section we consider the nonconforming finite element approximation to the thin plate model \eqref{p1} whose solution $u_n\in H^2(\Om)$ satisfies the following weak formulation
\beq\label{b1}
\lambda_n a_\Om(u_n,v)+(u_n,v)_n =(y,v)_n, \ \forall v\in H^2(\Omega).
\eeq
We assume $\Om$ is a polygonal or polyhedral domain in $\R^d$ $(d=2,3)$ in the reminder of this paper. Let $\cM_h$ be a family of shape regular and quasi-uniform finite element meshes over the domain $\Om$. We will use the Morley element \cite{Morley} for 2D, \cite{wang} for 3D to
define our nonconforming finite element method. The Morley element is a triple $(K,P_K,\Sigma_K)$, where $K\in\cM_h$ is a simplex in $\R^d$, $P_K=P_2(K)$ is the set of second order polynomials in $K$, and $\Sigma_K$ is the set of the degrees of freedom. In 2D, for the element $K$ with vertices $a_i, 1\le i\le 3$, and mid-points $b_i$ of the edge opposite to the vertex $a_i$, $1\le i\le 3$, $\Sigma_K=\{p(a_i), \pa_\nu p(b_i), 1\le i\le 3,\forall p\in C^1(K)\}$. In 3D, for the element $K$ with edges $S_{ij}$ which connects the vertices $a_i,a_j$, $1\le i<j\le 4$,
and faces $F_j$ opposite to $a_j$, $1\le j\le 4$, $\Sigma_K=\{\frac 1{|S_{ij}|}\int_{S_{ij}}p, 1\le i<j\le 4,
\frac 1{|F_j|}\int_{F_j}\pa_\nu p,1\le j\le 4,\forall p\in C^1(K)\}$. Here $\pa_\nu p$ is the normal derivative of $p$ of the edges (2D)
or faces (3D) of the element. We refer to Figure \ref{c1} for the illustration of the degrees of freedom of the Morley element.

\begin{figure}[t]
\begin{center}
\includegraphics[width=10cm]{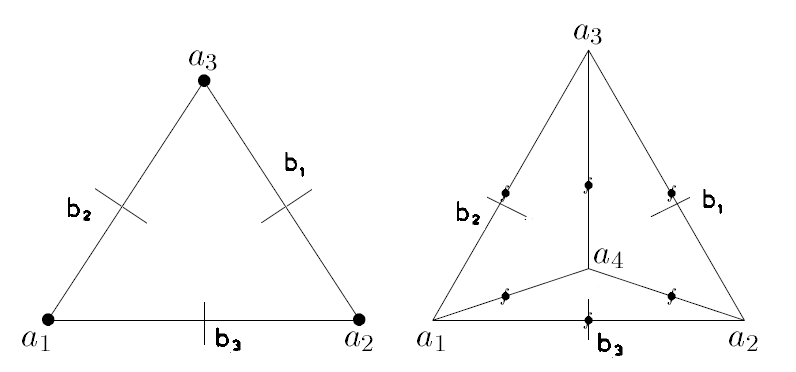}
\end{center}
\caption{The degrees of freedom of 2D Morley (left) and 3D Morley (right) element.}\label{c1}
\end{figure}

Let $V_h$ be the Morley finite element space 
\ben
V_h=\{v_h: v_h|_K\in P_2(K),\forall K\in\cM_h,f(v_h|_{K_1})=f(v_h|_{K_2}),\forall f\in\Sigma_{K_1}\cap\Sigma_{K_2}\}.
\een
The functions in $V_h$ may not be continuous in $\Om$. Given a set $G\subset\R^2$, let $\cM_h(G)=\{K\in\cM_h: G\cap K\not=\emptyset\}$ and $N(G)$ the number of elements in $\cM_h(G)$. For any $v_h\in V_h$, we define
\beq\label{h1-1}
\hat v_h(x_i)=\frac 1{N(x_i)}\sum_{K'\in\cM_h(x_i)} (v_h|_{K'})(x_i),\ \ i=1,2,\cdots,n.
\eeq
Notice that if $x_i$ is located inside some element $K$, then $\cM_h(x_i)=\{K\}$ and $\hat v_h(x_i)=v_h(x_i)$, $i=1,2,\cdots,n$.
With this definition we know that $(\hat v_h,\hat w_h)_n$ and $(e,\hat w_h)_n$ are well-defined for any $v_h,w_h\in V_h$.

Let
\ben
a_h(u_h,v_h)=\sum_{K\in\cM_h}\sum_{1\le i,j\le d}\int_K\frac{\pa^2 u_h}{\pa x_i\pa x_j}\frac{\pa^2 v_h}{\pa x_i\pa x_j}dx,\ \ \ \ \forall u_h,v_h\in V_h.
\een
The finite element approximation of the problem \eqref{b1} is to find $u_h\in V_h$ such
that
\beq\label{d1}
\lambda_n a_h(u_h,v_h)+(\hat u_h,\hat v_h)_n=(y,\hat v_h)_n,\ \ \ \ \forall v_h\in V_h.
\eeq
Since the sampling point set $\mathbb{T}$ is not collinear, by Lax-Milgram lemma, the problem \eqref{d1} has a unique solution.

Let $I_K: H^2(K)\to P_2(K)$ be the canonical local nodal value interpolant of Morley element \cite{Shi, wang} and $I_h: L^2(\Om)\to V_h$ be the global nodal value interpolant such that $(I_h u)|_K=I_K u$ for any $K\in\cM_h$ and piecewise $H^2(K)$ functions $u\in L^2(\Om)$. We introduce the mesh dependent semi-norm $|\cdot|_{m,h}$, $m\ge 0$,
\ben
|v|_{m,h}=\left(\sum_{K\in\cM_h}|v|_{H^m(K)}^2\right)^{1/2},
\een
for any $v\in L^2(\Om)$ such that $v|_K\in H^m(K),\forall K\in\cM_h$.
\begin{lemma}\label{lem:3.1}
We have
\be
& &|u-I_K u|_{H^m(K)}\le Ch_K^{2-m}|u|_{H^2(K)},\ \ \forall u\in H^m(K), 0\le m\le 2, \label{b3}\\
& &\|u-\widehat{I_h u}\|_n\le Ch^2|u|_{H^2(\Om)},\ \ \forall u\in H^2(\Om),\label{b4}
\ee
where $h_K$ is the diameter of the element $K$ and $h=\max_{K\in\cM_h} h_K$.
\end{lemma}

\begin{proof} Since $I_K p=p$ for any $p\in P_2(K)$ \cite{wang}, the estimate \eqref{b3} follows from the standard interpolation theory for finite element method \cite{Ciarlet}. Moreover, we have, by local inverse estimates and the standard interpolation estimates
\ben
\|u-I_K u\|_{L^\infty(K)}&\le&\inf_{p\in P_2(K)}\left[\|u-p\|_{L^\infty(K)}+|K|^{-1/2}\|I_K(u-p)\|_{L^2(K)}\right]\\
&\le& Ch_K^{2-d/2}|u|_{H^2(K)}.
\een
Let $\mathbb{T}_K=\{x_i\in\mathbb{T}: x_i\in K, 1\le i\le n\}$. By the assumption $\mathbb{T}$ is uniformly distributed and the mesh is quasi-uniform, we know that the cardinal $\#\mathbb{T}_K\le Cnh^d$. Thus
\ben
\|u-\widehat{I_h u}\|_n^2\le\frac 1n\sum_{K\in\cM_h}\#\mathbb{T}_K\|u-I_K u\|_{L^\infty(K)}^2\le Ch^4|u|_{H^2(\Om)}^2.
\een
This proves \eqref{b4}.
\end{proof}

The following property of Morley element will be used below.
\begin{lemma}\label{lem:new}
Let $K,K'\in\cM_h$ and $F=K\cap K'$.  There exists a constant C independent of $h$ such that for any $v_h\in V_h$, $|\alpha|\le 2 $,
\ben
\|\pa^{\alpha}(v_h|_K-v_h|_{K'})\|_{L^{\infty}(F)}\leq Ch^{2-|\alpha|-d/2}(|v_h|_{H^2(K)}+|v_h|_{H^2(K')}).
\een
\end{lemma}

\begin{proof}
By \cite[Lemma 5]{wang} we know that  
\ben
\|v_h|_K-v_h|_{K'}\|_{L^2(F)}\leq Ch^{3/2} (|v_h|_{H^2(K)}+|v_h|_{H^2(K')}).
\een
By using the inverse estimate we then obtain
\ben
\|\pa^{\alpha}(v_h|_K-v_h|_{K'})\|_{L^{\infty}(F)} & \leq &  C h^{-|\alpha|}\|v_h|_K-v_h|_{K'}\|_{L^{\infty}(F)} \\
& \leq & Ch^{-|\alpha|-(d-1)/2}\|v_h|_K-v_h|_{K'}\|_{L^2(F)}\\
&\leq & Ch^{2-|\alpha|-d/2}(|v_h|_{H^2(K)}+|v_h|_{H^2(K')}).
\een
This proves the lemma.
\end{proof}

\begin{lemma}\label{lem:3.2} There exists a linear operator $\Pi_h:V_h\to H^2(\Om)$ such that for any $v_h\in V_h$,
\be
& &|v_h-\Pi_hv_h|_{m,h}\le Ch^{2-m}|v_h|_{2,h},\ \ m=0,1,2,\label{b7}\\
& &\|\hat v_h-\Pi_h v_h\|_n\le Ch^2|v_h|_{2,h},\label{b5}
\ee
where the constant $C$ is independent of $h$.
\end{lemma}

\begin{proof} We will only prove the lemma for the case $d=2$. The case of $d=3$ will be briefly discussed in the appendix of this paper. We will construct $\Pi_hv_h$ by using the Agyris element. We recall \cite[P.71]{Ciarlet} that for any $K\in\cM_h$, Agyris element is a triple $(K,P_K,\Lambda_K)$, where $P_K=P_5(K)$ and the set of degrees of freedom, with the notation in Figure \ref{c2}, $\Lambda_K=\{p(a_i), Dp(a_i)(a_j-a_i), D^2p(a_i)(a_j-a_i,a_k-a_i), \pa_\nu p(b_i),1\le i,j,k\le 3, j\not=i, k\not=i, \forall p\in C^2(K)\}$. Let $X_h$
 be the Agyris finite element space 
\ben
X_h=\{v_h: v_h|_K\in P_5(K), \forall K\in\cM_h, f(v_h|_{K_1})=f(v_h|_{K_2}),\forall f\in \Lambda_{K_1}\cap\Lambda_{K_2}.\}
\een
It is known that $X_h\subset H^2(\Om)$.

\begin{figure}[t]
\begin{center}
\includegraphics[width=10cm]{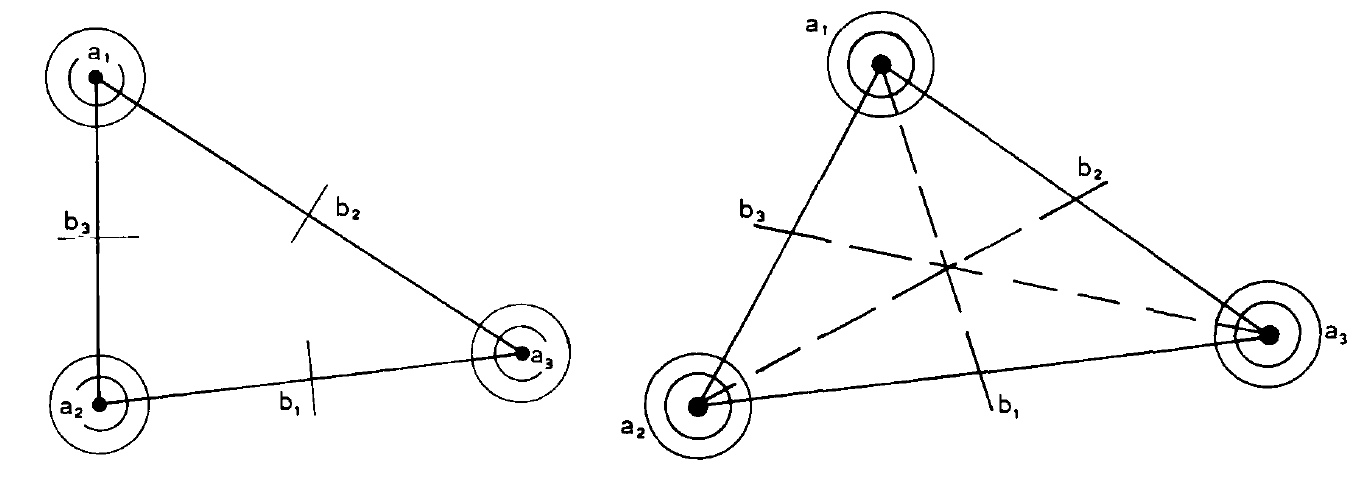}
\end{center}
\caption{The degrees of freedom of Agyris element (left) and Hermite triangle of type (5) (right).}\label{c2}
\end{figure}

We define the operator $\Pi_h$ as follows. For any $v_h\in V_h$,  $w_h:=\Pi_h v_h\in X_h$  such that for any $K\in\cM_h$, $w_h|_K\in P_5(K)$ and
\be
& &\pa^\alpha (w_h|_K)(a_i)=\frac 1{N(a_i)}\sum_{K'\in\cM_h(a_i)} \pa^\alpha(v_h|_{K'})(a_i),\ \ 1\le i\le 3, \ \ |\alpha|\le 2,\label{n1}\\
& &\pa_\nu (w_h|_K)(b_i)=\pa_\nu (v_h|_K)(b_i),\ \ 1\le i\le 3.\label{n2}
\ee
Here $\cM_h(a_i)$ and $N(a_i)$ are defined above \eqref{h1-1}.
To show the estimate \eqref{b7} we follow an idea in \cite[Theorem 6.1.1]{Ciarlet} and use the element Hermite triangle of type (5) \cite[P.102]{Ciarlet}, which is a triple $(K,P_K,\Theta_K)$, where $P_K=P_5(K)$ and the set of degrees of freedom $\Theta_K=\{p(a_i), Dp(a_i)(a_j-a_i), D^2p(a_i)(a_j-a_i,a_k-a_i), Dp(b_i)(a_i-b_i), 1\le i,j,k\le 3, j\not=i,k\not=i, \forall p\in C^2(K)\}$. The finite element space of Hermite triangle of type (5) is $H^1$ conforming and a regular family of Hermite triangle of type (5) is affine-equivalent. For any $K\in\cM_h$, denote by $p_i, p_{ij}, p_{ijk},q_i$ the basis functions associated with the degrees of freedom $p(a_i), Dp(a_i)(a_j-a_i), D^2p(a_i)(a_j-a_i,a_k-a_i), Dp(b_i)(a_i-b_i)$, $1\le i,j,k\le 3, j\not=i, k\not=i$. 
 
For any $v_h\in V_h$, we also define a linear operator $q_h:=\Lambda_h v_h$ as follows: for any $K\in\cM_h$, $q_h|_K\in P_5(K)$ and
\be
& &\pa^\alpha (q_h|_K)(a_i)=\frac 1{N(a_i)}\sum_{K'\in\cM_h(a_i)} \pa^\alpha(v_h|_{K'})(a_i),\ \ 1\le i\le 3, \ \ |\alpha|\le 2,\label{n3}\\
& &D(q_h|_K)(b_i)(a_i-b_i)=D(v_h|_K)(b_i)(a_i-b_i),\ \ 1\le i\le 3.\label{n4}
\ee 
Then from the definition of Morley element and Hermite triangle of type (5), we know that
$\phi_h|_K:=(v_h-q_h)|_K\in P_5(K)$ satisfies
\ben
\phi_h(x)&=&\sum_{i,j=1,2,3,j\not= i}D(\phi_h|_K)(a_i)(a_j-a_i)p_{ij}(x)\\
&+&\sum_{i,j,k=1,2,3,j\not=i,k\not=i}D^2(\phi_h|_K)(a_i)(a_j-a_i,a_k-a_i)p_{ijk}(x).
\een
Since a regular family of Hermite triangle of type (5) is affine-equivalent, by standard scaling argument \cite[Theorem 3.1.2]{Ciarlet}, we obtain easily $|q_i|_{H^m(K)}+|p_i|_{H^m(K)}+|p_{ij}|_{H^m(K)}+|p_{ijk}|_{H^m(K)}\le Ch_K^{1-m}$, $m=0,1,2$.
Thus, for $m=0,1,2$,
\beq\label{b6}
|\phi_h|_{H^m(K)}\le Ch_K^{1-m}\left(\sum_{i=1}^3\sum_{1\le |\alpha|\le 2}h^{|\alpha|}|\pa^\alpha(v_h|_K)(a_i)-\pa^\alpha (q_h|_K)(a_i)|^2\right)^{1/2}.
\eeq
By Lemma \ref{lem:new} and the fact that $\pa^\alpha (q_h|_K)(a_i)$ is the local average of $\pa^\alpha v_h$ over elements around $a_i$ in \eqref{n3}
\ben
|\pa^\alpha(v_h|_K)(a_i)-\pa^\alpha (q_h|_K)(a_i)|\le Ch^{1-|\alpha|}\left(\sum_{K'\in\cM_h(a_i)}|v_h|_{H^2(K')}^2\right)^{1/2},\ \ \ \ \forall 1\le |\alpha|\le 2.
\een
Inserting above estimate into \eqref{b6}, we get 
\beq\label{b8}
|v_h-q_h|_{H^m(K)}\le Ch^{2-m}\left(\sum_{K'\in\cM_h(K)}|v_h|_{H^2(K')}^2\right)^{1/2},\ \ m=0,1,2. 
\eeq
By \eqref{n1}-\eqref{n4} we know that $q_h-w_h\in P_5(K)$ and satisfies 
\ben
q_h(x)-w_h (x)&=& \sum\limits_{i=1}^3 D(q_h|_K-w_h|_K)(b_i)(a_i-b_i) q_i(x).
\een
On the other hand, for $1\le i\le 3$,
\ben
D(q_h|_K-w_h|_K)(b_i)(a_i-b_i)=\pa_\nu(q_h|_K-v_h|_K)(b_i)[(a_i-b_i)\cdot\nu],
\een
since $\pa_\nu (w_h|_K)(b_i)=\pa_\nu(v_h|_K)(b_i)$ by \eqref{n2} and the tangential derivative of $(q_h|_K-w_h|_K)$ vanishes as a consequence of \eqref{n1} and \eqref{n3}. Since $|q_i|_{H^m(K)}\le Ch_K^{1-m}$ for $m=0,1,2$, we obtain then
\be
|q_h-w_h|_{H^m(K)}&\le&Ch^{2-m}\left(\sum^3_{i=1}|\pa_\nu(q_h|_K-v_h|_K)(b_i)|^2\right)^{1/2}\nn\\
&\le&Ch^{2-m}\left(\sum_{K'\in\cM_h(K)}|v_h|_{H^2(K')}^2\right)^{1/2},\ \ m=0,1,2,\label{n5}
\ee
where in the second inequality we have used the fact that by the inverse estimate and \eqref{b8},
\ben
|\pa_\nu(q_h|_K-v_h|_K)(b_i)|\le |q_h-v_h|_{W^{1,\infty}(K)}&\le&Ch^{-1}_K|q_h-v_h|_{H^1(K)}\\
&\le&C\left(\sum_{K'\in\cM_h(K)}|v_h|_{H^2(K')}^2\right)^{1/2}.
\een  
Combining \eqref{b8} and \eqref{n5} shows \eqref{b7}.

To show \eqref{b5}, we use the notation in the proof of Lemma \ref{lem:3.1}, the inverse estimate and \eqref{b7} to get
\ben
\|\hat v_h-w_h\|_n^2\le\frac Cn\sum_{K\in\cM_h}\#\mathbb{T}_K \|v_h-w_h\|_{L^\infty(K)}^2\le C\|v_h-w_h\|^2_{L^2(\Om)}\le Ch^4|v_h|_{2,h}^2.
\een
This completes the proof.
\end{proof}

For any function $v$ which is piecewise in $C^2(K)$ for any $K\in\cM_h$, we use the convenient energy norm
\ben
\lj v\rj_{h}=\left(\lambda_n |v|_{2,h}^2+\|\hat v\|_n^2\right)^{1/2}.
\een
Here $\hat v(x_i)$, $i=1,2,\cdots,n$, is defined as in \eqref{h1-1}, that is, $\hat v(x_i)$ is the local average of all $v|_{K'}(x_i)$, where $K'\in\cM_h$ such that $x_i\in K'$.

\begin{theorem}\label{thm:3.1}
Let $u_n\in H^2(\Om)$ be the unique solution of \eqref{b1} and $u_h\in V_h$ be the solution of \eqref{d1}.
Then there exist constants $\lambda_0 > 0$ and $C>0$ such that for any $\lambda_n \leq \lambda_0$ and $n\lambda_n^{d/4}\ge 1$,
\beq\label{d2}
\mathbb{E}\big[\|u_0-\hat u_h\|_n^2\big]\le C(\lambda_n+h^4)|u_0|^2_{H^2(\Om)}+C\left[1+\frac{h^4}{\lam_n}+\left(\frac{h^4}{\lambda_n}\right)^{1-d/4}\right]\frac{\sigma^2}{n\lambda_n^{d/4}}.
\eeq
In particular, if $h^4\le C\lambda_n$, we have
\beq\label{d3}
\mathbb{E}\big[\|u_0-\hat u_h\|_n^2\big]\le C\lambda_n|u_0|_{H^2(\Om)}^2+\frac{C\sigma^2}{n\lambda_n^{d/4}}.
\eeq
\end{theorem}

\begin{proof} We start by using the Strang lemma \cite{Ciarlet}
\beq\label{d4}
\lj u_n-\hat u_h\rj_h\le C\inf_{v_h\in V_h}\lj u_n-\hat v_h\rj_h+C\sup_{0\not= v_h\in V_h}\frac{|\lam_n a_h(u_n,v_h)+(u_n-y,\hat v_h)_n|}{\lj v_h\rj_h}.
\eeq
By Lemma \ref{lem:3.1} we have
\beq\label{d5}
\inf_{v_h\in V_h}\lj u_n-\hat v_h\rj_h\le C(\lam_n^{1/2}+h^2)|u_n|_{H^2(\Om)}.
\eeq
Since for any $v_h\in V_h$, $\Pi_hv_h\in H^2(\Om)$, by \eqref{b1} and the fact that $y_i=u_0(x_i)+e_i$, $i=1,2,\cdots n$, to obtain
\ben
& &\lam_n a_h(u_n,v_h)+(u_n-y,\hat v_h)_n\\
&=&\lam_n a_h(u_n,v_h-\Pi_hv_h)+(u_n-y,\hat v_h-\Pi_hv_h)_n\\
&\le&\lam_n|u_n|_{H^2(\Om)}|v_h-\Pi_hv_h|_{2,h}+\|u_n-u_0\|_n\|\hat v_h-\Pi_hv_h\|_n+(e,\hat v_h-\Pi_hv_h)_n.
\een
Now by using Lemma \ref{lem:3.2} we have
\be\label{d6}
& &\sup_{0\not= v_h\in V_h}\frac{|\lam_n a_h(u_n,v_h)+(u_n-y,\hat v_h)_n|}{\lj v_h\rj_h}\nn\\
&\le&C\lam_n^{1/2}|u_n|_{H^2(\Om)}+C\frac{h^2}{\lam_n^{1/2}}\|u_n-u_0\|_n+\sup_{0\not= v_h\in V_h}\frac{|(e,\hat v_h-\Pi_hv_h)_n|}{\lj v_h\rj_{h}}.
\ee
Since $e_i$, $i=1,2,\cdots, n$, are independent and identically distributed random variables, we have
\ben
\mathbb{E}\big[|(e,\hat v_h-\Pi_hv_h)_n|^2\big]=\sigma^2 n^{-1}\|\hat v_h-\Pi_hv_h\|_n^2\le C\sigma^2 n^{-1}h^4|v_h|_{2,h}^2,
\een
where we have used Lemma \ref{lem:3.2} in the last inequality. 

Let $N_h$ be the dimension of $V_h$ which satisfies $N_h\le Ch^{-d}$ since the mesh is quasi-uniform. Recall that if $\{X_i\}_{i=1}^{N_h}$ are random variables, $\E[\sup_{1\le i\le N_h}|X_i|]\le\sum^{N_h}_{i=1}\E [|X_i|]$. We have then
\beq\label{d7}
\mathbb{E}\left[\sup_{0\not= v_h\in V_h}\frac{|(e,v_h-\Pi_hv_h)_n|^2}{\lj v_h\rj_{h}^2}\right]\le N_h\cdot \sup_{0\not= v_h\in V_h}\mathbb{E}\left[\frac{|(e,v_h-\Pi_hv_h)_n|^2}{\lj v_h\rj_{h}^2}\right]\le C\frac{\sigma^2 h^{4-d}}{n\lam_n}.
\eeq
Combining \eqref{d4}-\eqref{d7} we obtain
\ben
\mathbb{E}\big[\lj u_n-\hat u_h\rj_h^2\big]\le C\lam_n\mathbb{E}\big[|u_n|_{H^2(\Om)}^2\big]+C\frac{h^4}{\lam_n}\mathbb{E}\big[\|u_n-u_0\|_n^2\big]
+C\frac{\sigma^2 h^{4-d}}{n\lam_n}.
\een
This completes the proof by using Theorem \ref{thm:2.1}.
\end{proof}

\section{Stochastic convergence}

In this section we study the stochastic convergence of the error $\|u_0-\hat u_h\|_n$ which characterizes the tail property of $\mathbb{P}(\|u_0-\hat u_h\|_n\ge z)$ for $z>0$. We assume the noises $e_i$, $i=1,2,\cdots,n$, are
independent and identically distributed sub-Gaussian random variables with parameter $\sigma>0$. A random variable $X$ is sub-Gaussion with parameter $\sigma$ if it satisfies
\beq\label{e1}
\mathbb{E}\left[e^{\lam(X-\mathbb{E}[X])}\right]\le e^{\frac 12\sigma^2\lam^2},\ \ \ \ \forall\lam\in\R.
\eeq
The probability distribution function of a sub-Gaussion random variable has a exponentially decaying tail, that is, if $X$ is a sub-Gaussion random variable, then
\beq\label{gg1}
\mathbb{P}(|X-\E [X]|\ge z)\le 2e^{-\frac 12 z^2/\sigma^2},\ \ \forall z>0.
\eeq
In fact, by Markov inequality, for any $\lam>0$,
\ben
\mathbb{P}(X-\E [X]\ge z)=\mathbb{P}(\lam(X-\E [X])\ge \lam z)\le e^{-\lam z}\E [e^{\lam (X-\E [X])}]\le e^{-\lam z-\frac 12\sigma^2\lam^2}.
\een
By taking $\lam=z/\sigma^2$ yields $\mathbb{P}(X-\E [X]\ge z)\le e^{-\frac 12 z^2/\sigma^2}$. Similarly, one can prove $\mathbb{P}(X-\E [X]\le -z)\le e^{-\frac 12 z^2/\sigma^2}$. This shows \eqref{gg1}.

\subsection{Stochastic convergence of the thin plate splines}
We will use several tools from the theory of empirical processes \cite{Vaart, Geer} for our analysis. We start by recalling the definition of Orlicz norm. Let $\psi$ be a monotone increasing convex function satisfying $\psi(0)=0$.
Then the Orilicz norm $\|X\|_\psi$ of a random variable $X$ is defined as
\beq\label{e2}
\|X\|_\psi=\inf\left\{C>0:\mathbb{E}\left[\psi\left(\frac{|X|}C\right)\right]\le 1\right\}.
\eeq
By using Jensen inequality, it is easy to check $\|X\|_\psi$ is a norm. In the following we will use the $\|X\|_{\psi_2}$ norm with
$\psi_2(t)=e^{t^2}-1$ for any $t>0$. By definition we know that
\beq\label{e3}
\mathbb{P}(|X|\ge z)\le 2\,e^{-z^2/\|X\|_{\psi_2}^2},\ \ \ \ \forall z>0.
\eeq
The following lemma is from \cite[Lemma 2.2.1]{Vaart} which shows the inverse of this property.
\begin{lemma}\label{lem:4.1}
If there exist positive constants $C,K$ such that $\mathbb{P}(|X|>z)\le Ke^{-Cz^2},\ \forall z>0$, then $\|X\|_{\psi_2}\le\sqrt{(1+K)/C}$.
\end{lemma}

Let $T$ be a semi-metric space with the semi-metric $d$ and $\{X_t:t\in T\}$ be a random process indexed by $T$. Then the random process $\{X_t:t\in T\}$ is called sub-Gaussian if
\beq\label{e41}
\mathbb{P}(|X_s-X_t|>z)\le 2e^{-\frac 12 z^2/d(s,t)^2},\ \ \ \ \forall s,t\in T, \ \ z>0.
\eeq
For a semi-metric space $(T,d)$, an important quantity to characterize the complexity of the set $T$ is the entropy which we now introduce. The covering number $N(\vep,T,d)$ is the minimum number of $\vep$-balls that cover $T$. A set is called $\vep$-separated if the distance of any two points in the set is strictly greater than $\vep$. The packing number $D(\vep,T,d)$ is the maximum number of $\vep$-separated points in $T$. $\log N(\vep,T,d)$ is called the covering entropy and $\log D(\vep,T,d)$ is called the packing entropy. It is easy to check that \cite[P.98]{Vaart}
\beq\label{f1}
N(\vep,T,d)\le D(\vep,T,d)\le N(\frac\vep 2,T,d).
\eeq

The following maximal inequality \cite[Section 2.2.1]{Vaart} plays an important role in our analysis.
\begin{lemma}\label{lem:4.2}
If $\{X_t:t\in T\}$ is a separable sub-Gaussian random process, then 
\ben
\|\sup_{s,t\in T}|X_s-X_t|\|_{\psi_2}\le K\int^{\diam\, T}_0\sqrt{\log D(\vep,T,d)}\ d\vep.
\een
Here $K>0$ is some constant.
\end{lemma}

The following result on the estimation of the entropy of Sobolev spaces is due to Birman-Solomyak \cite{Birman}.
\begin{lemma}\label{lem:4.3}
Let $Q$ be the unit square in $\R^d$ and $SW^{\alpha,p}(Q)$ be the unit sphere of the Sobolev space $W^{\alpha,p}(Q)$,
where $\alpha> 0, p\ge 1$. Then for $\vep>0$ sufficient small, the entropy
\ben
\log N(\vep, SW^{\alpha,p}(Q), \|\cdot\|_{L^q(Q)})\le C\vep^{-d/\alpha},
\een
where if $\alpha p>d$, $1\le q\le\infty$, otherwise if $\alpha p\le d$, $1\le q\le q^*$ with $q^*=p(1-\alpha p/d)^{-1}$.
\end{lemma}

For any $\de>0, \rho>0$, define
\beq\label{f3}
S_{\de,\rho}(\Om):=\{u\in H^2(\Om): \|u\|_n\le\de,|u|_{H^2(\Om)}\le\rho\}.
\eeq
The following lemma estimates the entropy of the set $S_{\de,\rho}(\Om)$.
\begin{lemma}\label{lem:4.4} There exists a constant $C$ independent of $\de,\rho,\vep$ such that
\ben
\log N(\vep, S_{\de,\rho}(\Om), \|\cdot\|_{L^\infty(\Om)})\le C\left(\frac{\rho+\de}\vep\right)^{d/2}.
\een
\end{lemma}

\begin{proof} By \eqref{f2} we have for any $u\in S_{\de,\rho}(\Om)$, $\|u\|_{H^2(\Om)}\le C(\|u\|_{L^2(\Om)}+|u|_{H^2(\Om)})
\le C(\|u\|_n+|u|_{H^2(\Om)})\le C(\de+\rho)$, where we have used the fact that $h_{\max}\le Cn^{-1/d}\le C$. The lemma now
follows from Lemma \ref{lem:4.3}.
\end{proof}

The following lemma is proved by the argument in \cite[Lemma 2.2.7]{Vaart}.
\begin{lemma}\label{lem:4.5}
$\{E_n(u):=(e,u)_n: u\in H^2(\Om)\}$ is a sub-Gaussian random process with respect to the semi-distance $d(u,v)=\|u-v\|_n^*$, where $\|u\|^*_n:=\sigma n^{-1/2}\|u\|_n$.
\end{lemma}

\begin{proof} By definition $E_n(u)-E_n(v)=\sum^n_{i=1}c_ie_i$, where $c_i=\frac{1}{n}(u-v)(x_i)$. Since $e_i$ is a sub-Gaussion random variable with parameter $\sigma$ and $\E [e_i]=0$, by \eqref{e1}, $\E [e^{\lam e_i}]\le e^{\frac 12\sigma^2\lam^2}, \forall\lam>0$. Thus, since $e_i$, $i=1,2,\cdots,n$, are independent random variables,
\ben
\E\left[e^{\lam\sum^n_{i=1}c_ie_i}\right]\le e^{\frac 12\sigma^2\lam^2\sum^n_{i=1}c_i^2}=e^{\frac 12\sigma^2n^{-1}\lam^2\|u-v\|_n^2}=e^{\frac 12d(u,v)^2\lam^2}.
\een
This shows $E_n(u)-E_n(v)$ is a sub-Gaussion random variable with parameter $d(u,v)$. By \eqref{gg1} we have
\ben
\mathbb{P}(|E_n(u)-E_n(v)|\ge z)\le 2e^{-\frac 12z^2/d(u,v)^2},\ \ \forall z>0.
\een
This shows the lemma by the definition of sub-Gaussion random process \eqref{e41}.
\end{proof}

The following lemma which improves Lemma \ref{lem:4.1} will be used in our subsequent analysis.
\begin{lemma}\label{lem:4.6}
If $X$ is a random variable which satisfies
\ben
\mathbb{P}(|X|>\alpha (1+z))\le C_1e^{-z^2/K_1^2},\ \ \forall\alpha>0, z\ge 1 ,
\een
where $C_1, K_1$ are some positive constants, then $\|X\|_{\psi_2}\le C(C_1,K_1)\alpha$ for some constant $C(C_1,K_1)$ depending only on $C_1,K_1$.
\end{lemma}
\begin{proof} If $y\ge 2\alpha$, then $z=(y/\alpha)-1\ge 1$. Thus
\ben
\mathbb{P}(|X|>y)=\mathbb{P}(|X|>\alpha (1+z))\le C_1\exp \left[-\frac 1{K_1^2}\left(\frac y\alpha-1\right)^2\right].
\een
Since $(\frac y\alpha -1)^2\ge \frac 12(\frac y\alpha)^2-1$ by Cauchy-Schwarz inequality, we obtain
\ben
\mathbb{P}(|X|>y)\le C_1e^{\frac 1{K_1^2}}e^{-\frac{y^2}{2K_1^2\alpha^2}}=C_1e^{\frac 1{K_1^2}}e^{-\frac{y^2}{K^2_2}},
\een
where $K_2:=\sqrt 2\alpha K_1$. On the other hand, if $y<2\alpha$, then
\ben
\mathbb{P}(|X|>y)\le e^{\frac{y^2}{K_2^2}}e^{-\frac{y^2}{K_2^2}}\le e^{\frac{2}{K^2_1}}e^{-\frac{y^2}{K_2^2}}.
\een
Therefore, $\mathbb{P}(|X|>y)\le C_2e^{-y^2/K_2^2}$, $\forall y>0$, where $C_2=\max(C_1e^{1/{K_1^2}},e^{2/K_1^2})$. This implies by Lemma \ref{lem:4.1},
\ben
\|X\|_{\psi_2}\le\sqrt{1+C_2}K_2=C(C_1,K_1)\alpha,\ \ \mbox{where }C(C_1,K_1)=\sqrt 2K_1\sqrt{1+C_2}.
\een
This completes the proof.
\end{proof}

\begin{theorem}\label{thm:4.1}
Let $u_n\in H^2(\Om)$ be the solution of \eqref{b1}. Denote by $\rho_0=|u_0|_{H^2(\Omega)}+\sigma n^{-1/2}$. If we take
\beq\label{e4}
\lambda_n^{1/2+d/8}= O(\sigma n^{-1/2}\rho_0^{-1}),
\eeq
then there exists a constant $C>0$ such that
\beq\label{e5}
\|\,\|u_n-u_0\|_n\,\|_{\psi_2} \leq C\lambda_n^{1/2}\rho_0,\ \ \ \ \|\,|u_n|_{H^2(\Om)}\,\|_{\psi_2}\le C\rho_0.
\eeq
\end{theorem}
\begin{proof} We will only prove the first estimate in \eqref{e5} by the peeling argument. The other estimate can be proved in a similar way.
It follows from \eqref{p7} that
\beq\label{g1}
\|u_n-u_0\|^2_n + \lambda_n |u_n|_{H^2(\Omega)}^2 \leq 2(e,u_n-u_0)_n + \lambda_n |u_0|_{H^2(\Omega)}^2.
\eeq
Let $\delta >0,\ \rho>0$ be two constants to be determined later, and
\beq\label{g2}
A_0=[0,\delta), A_i=[2^{i-1}\delta,2 ^i\delta), \ \ B_0=[0,\rho), B_j=[2^{j-1}\rho,2^j\rho),\ \ \ \ i,j\ge 1.
\eeq
For $i,j\ge 0$, define
\ben
F_{ij}= \{v \in H^2(\Omega):~  \|v\|_n \in A_i ~,~ |v|_{H^2(\Omega)} \in B_j \}.
\een
Then we have
\beq\label{g3}
\mathbb{P}(\|u_n-u_0\|_n>\de)\le\sum_{i=1}^\infty\sum_{j=0}^\infty \mathbb{P}(u_n-u_0\in F_{ij}).
\eeq
Now we estimate $\mathbb{P}(u_n-u_0\in F_{ij})$. By Lemma \ref{lem:4.5}, $\{(e,v)_n:v\in H^2(\Om)\}$ is a sub-Gaussion random process with respect to the semi-distance $d(u,v)=\sigma n^{-1/2}\|u-v\|_n$. It is
easy to see that
\ben
\diam\, F_{ij}\le \sigma n^{-1/2}\sup_{u-u_0,v-u_0\in F_{ij}}(\|u-u_0\|_n+\|v-u_0\|_n)\le 2\sigma n^{-1/2}\cdot 2^i\de.
\een
Then by \eqref{f1} and the maximal inequality in Lemma \ref{lem:4.2} we have
\ben
\|\sup_{u-u_0\in F_{ij}}|(e,u-u_0)_n|\|_{\psi_2}&\le& K\int^{\sigma n^{-1/2}\cdot 2^{i+1}\de}_0\sqrt{\log N\left(\frac\vep 2,F_{ij}, d\right)}\,d\vep\nn\\
&=&K\int^{\sigma n^{-1/2}\cdot 2^{i+1}\de}_0\sqrt{\log N\left(\frac\vep{2\sigma n^{-1/2}},F_{ij}, \|\cdot\|_n\right)}\,d\vep.
\een
By Lemma
\ref{lem:4.4} we have the estimate for the entropy
\ben
\log N\left(\frac\vep{2\sigma n^{-1/2}},F_{ij}, \|\cdot\|_n\right)&\le&\log N(\frac\vep{2\sigma n^{-1/2}},F_{ij}, \|\cdot\|_{L^\infty(\Om)})\\
&\le&C\left(\frac{2\sigma n^{-1/2}\cdot(2^i\de+2^j\rho)}{\vep}\right)^{d/2}.
\een
Therefore,
\be
\|\sup_{u-u_0\in F_{ij}}|(e,u-u_0)_n\|_{\psi_2}&\le&K\int^{\sigma n^{-1/2}\cdot 2^{i+1}\de}_0\left(\frac{2\sigma n^{-1/2}\cdot(2^i\de+2^j\rho)}{\vep}\right)^{d/4}\,d\vep\nn\\
&=&C\sigma n^{-1/2}(2^i\de+2^j\rho)^{d/4}(2^i\de)^{1-d/4}\nn\\
&\le&C\sigma n^{-1/2}[2^i\de+(2^i\de)^{1-d/4}(2^j\rho)^{d/4}].\label{g5}
\ee
By \eqref{g1} and \eqref{e3} we have for $i,j \geq 1$:
\ben
\mathbb{P}(u_n-u_0\in F_{ij})& \leq & \mathbb{P}(2^{2(i-1)}\delta^2 + \lambda_n 2^{2(j-1)}\rho^2 \leq 2 \mathop {\sup}\limits_{u-u_0 \in F_{ij}}|(e,u-u_0)_n|
+ \lambda_n \rho^2_0 ) \\
& =& \mathbb{P}(2 \mathop {\sup}\limits_{u-u_0 \in F_{ij}}|(e,u-u_0)_n|\ge 2^{2(i-1)}\delta^2 + \lambda_n 2^{2(j-1)}\rho^2 - \lambda_n \rho^2_0)\\
&\le &2\exp \left[- \frac{1}{C\sigma^2 n^{-1}} \left(\frac{2^{2(i-1)}\delta^2 + \lambda_n 2^{2(j-1)}\rho^2-\lambda_n \rho^2_0}{2^i\delta + (2^i\delta)^{1-d/4} (2^j\rho)^{d/4}} \right)^2\right].
\een
Now we take
\beq\label{g6}
\delta^2=\lambda_n\rho_0^2 (1+z)^2,\ \rho=\rho_0, \ \ \mbox{where }z \ge 1.
\eeq
Since by assumption $\lambda_n^{1/2+d/8}=O(\sigma n^{-1/2}\rho_0^{-1})$ and $\sigma n^{-1/2}\rho_0^{-1}\leq 1$, we have
$\lam_n\le C$ for some constant. By some simple calculation we have for $i,j \geq 1$,
\ben
\mathbb{P}(u_n-u_0\in F_{ij})& \le & 2\exp \left[ - C \left(\frac{2^{2(i-1)}z(1+z) + 2^{2(j-1)}}{ 2^i(1+z) + (2^i (1+z))^{1-d/4} (2^j)^{d/4}} \right)^2\right].
\een
By using the elementary inequality $ab\le \frac 1p a^p+\frac 1q b^q$ for any $a,b>0, p,q>1, p^{-1}+q^{-1}=1$, we have
$(2^i (1+z))^{1-d/4} (2^j)^{d/4}\le (1+z)2^i+2^j$. Thus
\ben
\mathbb{P}(u_n-u_0\in F_{ij})\le 2\exp \left[ - C (2^{2i} z^2 + 2^{2j}) \right].
\een
Similarly, one can prove for $i\geq 1, j=0$,
\ben
\mathbb{P}(u_n-u_0\in F_{i0}) \le 2\exp \left[- C (2^{2i} z^2) \right].
\een
Therefore, since $\sum^\infty_{j=1}e^{-C(2^{2j})}\le e^{-C}< 1$ and $\sum^\infty_{i=1}e^{-C(2^{2i}z^2)}\le e^{-Cz^2}$, we obtain finally
\ben
\sum_{i=1}^\infty\sum_{j=0}^\infty \mathbb{P}(u_n-u_0\in F_{ij})&\le&2\sum_{i=1}^\infty\sum_{j=1}^\infty e^{- C (2^{2i} z^2 + 2^{2j})}+2\sum^\infty_{i=1}e^{- C (2^{2i} z^2)}\\
&\le&4e^{-Cz^2}.
\een
Now inserting the estimate to \eqref{g3} we have
\beq\label{g7}
\mathbb{P}(\|u_n-u_0\|_n>\lam_n^{1/2}\rho_0(1+z))\le 4e^{-Cz^2}, \ \ \ \ \forall z\ge 1.
\eeq
This implies by using Lemma \ref{lem:4.6} that $\|\|u_n-u_0\|_n\|_{\psi_2} \leq C \lambda_n^{1/2}\rho_0$.
This completes the proof.
\end{proof}

We remark that \eqref{g7} implies that
\ben
\lim_{z\to\infty}\overline{\lim_{n\to\infty}}\ \mathbb{P}(\|u_n-u_0\|_n>\lam_n^{1/2}\rho_0(1+z))=0.
\een
In terms of the terminology of the stochastic convergence order, we have $\|u_n-u_0\|_n=O_p(\lam_n^{1/2})\rho_0$ which by the
assumption \eqref{e4} yields
\ben
\|u_n-u_0\|_n=O_p(n^{-\frac 2{4+d}})\sigma^{\frac 4{4+d}}\rho_0^{-\frac 4{4+d}}.
\een
This estimate is proved in \cite[Section 10.1.1]{Geer} when $d=1$. Our result in Theorem \ref{thm:4.1} is stronger in the sense that it also provides the tail property of the probability distribution function of the random error $\|u_n-u_0\|_n$.

\subsection{Stochastic convergence of the finite element method}
The following lemma provides the estimate of the entropy of finite dimension subsets \cite[Corollary 2.6]{Geer}.
\begin{lemma}\label{lem:4.7}
Let $G$ be a finite dimensional subspace of $L^2(\Om)$ of dimension $N>0$ and $G_R=\{f\in G: \|f\|_{L^2(\Om)}\le R\}$. Then
\ben
N(\vep,G_R,\|\cdot\|_{L^2(\Om)})\le (1+4R/\vep)^N,\ \ \forall \vep>0.
\een
\end{lemma}

\begin{lemma}\label{lem:4.8}
Let $G_h:=\{v_h\in V_h: \lj v_h\rj_h=(\lam_n|v_h|_{2,h}^2+\|\hat v_h\|_n^2)^{1/2}\le 1\}$. Assume that $h=O(\lam_n^{1/4})$ and $n\lam_n^{d/4}\ge 1$. Then
\ben
\|\,\sup_{v_h\in G_h}|(e,\hat v_h-\Pi_h v_h)_n|\,\|_{\psi_2}\le C\sigma n^{-1/2}\lam_n^{-d/8}.
\een
\end{lemma}

\begin{proof} Similar to the proof of Lemma \ref{lem:4.5} we know that $\{\hat E_n(v_h):=(e,\hat v_h-\Pi_h v_h)_n,\ \forall v_h\in G_h\}$ is a sub-Gaussion random process with respect to the semi-distance $\hat d(v_h,w_h)=\sigma n^{-1/2}\|(\hat v_h-\Pi_h v_h)-(\hat w_h-\Pi_h w_h)\|_n$. By Lemma \ref{lem:3.2}, for any $v_h\in G_h$, $\|\hat v_h-\Pi_hv_h\|_n\le Ch^2|v_h|_{2,h}\le Ch^2\lam_n^{-1/2}\le C$, where we have used the assumption $h=O(\lam_n^{1/4})$ in the last inequality. This implies that the
diameter of $G_h$ is bounded by $C\sigma n^{-1/2}$.
Now by the maximal inequality in Lemma \ref{lem:4.2}
\beq\label{i1}
\|\,\sup_{v_h\in G_h}|(e,\hat v_h-\Pi_h v_h)_n|\,\|_{\psi_2}\le K\int_0^{C\sigma n^{-1/2}}\sqrt{\log N\left(\frac\vep 2,G_h,\hat d\right)}\ d\vep.
\eeq
For any $v_h\in V_h$, by Lemma \ref{lem:3.2}, $\Pi_h v_h\in H^2(\Om)$ and thus by \eqref{f2}
\ben
\|\Pi_hv_h\|_{L^2(\Om)}&\le&C(h^2_{\max} |\Pi_hv_h|_{H^2(\Om)}+\|\Pi_hv_h\|_n)\\
&\le&C(n^{-2/d}\lam_n^{-1/2}+\|\Pi_hv_h-\hat v_h\|_n+\|\hat v_h\|_n)\\
&\le&C(n^{-2/d}\lam_n^{-1/2}+Ch^2\lam_n^{-1/2}+1)\\
&\le&C,
\een
where we have used $h=O(\lam_n^{1/4})$ and $n\lam_n^{d/4}\ge 1$ in the last inequality. Thus
\beq\label{i2}
\|v_h\|_{L^2(\Om)}\le \|v_h-\Pi_h v_h\|_{L^2(\Om)}+\|\Pi_h v_h\|_{L^2(\Om)}\le Ch^2|v_h|_{2,h}+C\le C,\ \ \ \ \forall v_h\in G_h.
\eeq
Moreover, by Lemma \ref{lem:3.2} and inverse estimate,
\beq\label{i3}
\hat d(v_h,w_h)\le C\sigma n^{-1/2}h^2|v_h-w_h|_{2,h}\le C\sigma n^{-1/2}\|v_h-w_h\|_{L^2(\Om)},\ \ \ \ \forall v_h,w_h\in V_h.
\eeq
Now since the dimension of $V_h$ is bounded by $Ch^{-d}$, Lemma \ref{lem:4.7} together with \eqref{i2}-\eqref{i3} implies
\ben
\log N\left(\frac\vep 2,G_h,\hat d\right)&=&\log N\left(\frac\vep{C\sigma n^{-1/2}},G_h,\|\cdot\|_{L^2(\Om)}\right)\\
&\le&Ch^{-d}(1+{\sigma n^{-1/2}}/\vep).
\een
Inserting this estimate to \eqref{i1}
\ben
\|\,\sup_{v_h\in G_h}|(e,\hat v_h-\Pi_h v_h)_n|\,\|_{\psi_2}&\le&C\int_0^{C\sigma n^{-1/2}}\sqrt{Ch^{-d}(1+\sigma n^{-1/2}/\vep)}\,
d\vep\\
&\le&Ch^{-d/2}\sigma n^{-1/2}.
\een
This completes the proof since $h=O(\lam_n^{1/4})$.
\end{proof}

The following theorem is the main result of this section.
\begin{theorem}\label{thm:4.2}
Let $u_h\in V_h$ be the solution of \eqref{d1}. Denote by $\rho_0=|u_0|_{H^2(\Omega)}+\sigma n^{-1/2}$. If we take
\beq\label{h1}
h=O(\lam_n^{1/4})\ \ \mbox{and }\ \lambda_n^{1/2+d/8}= O(\sigma n^{-1/2}\rho_0^{-1}),
\eeq
then there exists a constant $C>0$ such that
\beq\label{h2}
\|\,\|\hat u_h-u_0\|_n\,\|_{\psi_2} \leq C\lambda_n^{1/2}\rho_0,\ \ \ \ \|\,|u_h|_{H^2(\Om)}\,\|_{\psi_2}\le C\rho_0.
\eeq
\end{theorem}
\begin{proof} By \eqref{d4}-\eqref{d6} we have
\ben
& &\lam_n^{1/2}|u_h|_{H^2(\Om)}+\|\hat u_h-u_0\|_n\\
&\le&C(1+\frac{h^2}{\lam_n^{1/2}})\|u_n-u_0\|_n+C(h^2+\lam_n^{1/2})(|u_n|_{H^2(\Om)}+|u_0|_{H^2(\Om)})\\
&+&C\sup_{0\not= v_h\in V_h}\frac{|(e,\hat v_h-\Pi_h v_h)_n|}{\lj v_h\rj_h}.
\een
The theorem follows now from Theorem \ref{thm:4.1}, Lemma \ref{lem:4.8} and the assumption $\sigma n^{-1/2}\le C\lam_n^{1/2+d/8}\rho_0$.
\end{proof}

By \eqref{e3}, we know from Theorem \ref{thm:4.2} that
\ben
\mathbb{P}(\|\hat u_h-u_0\|_n\ge z)\le 2e^{-z^2/(C\lam_n\rho_0^2)},\ \ \ \ \forall z>0,
\een
that is, the probability density function of the random error $\|\hat u_h-u_0\|_n$ decays exponentially as $n\to\infty$.
\section{Numerical examples}

From Theorem \ref{thm:4.2} we know that the mesh size should be comparable with $\lam_n^{1/4}$. The smoothing parameter $\lam_n$ is usually determined by the cross-validation in the literature \cite{Wahba}. Here we propose a self-consistent algorithm to determine the parameter $\lam_n$
based on $\lam_n^{1/2+d/8}=\sigma n^{-1/2}(|u_0|_{H^2(\Om)}+\sigma n^{-1/2})^{-1}$ as indicated in Theorem \ref{thm:4.2}.
In the algorithm we estimate $|u_0|_{H^2(\Om)}$ by $|u_h|_{2,h}$ and $\sigma$ by $\|u_h-y\|_n$ since $\|u_0-y\|_n=\|e\|_n$ provides a good estimation of the variance by the law of large number.

\begin{algorithm}  {\sc(Self-consistent algorithm for finding $\lam_n$)} \label{j1}\\
$1^\circ$ Given an initial guess of $\lambda_{n,0}$; \\
$2^\circ$ For $k \geq 0$ and $\lambda_{n,k}$ known, compute $u_h$ with the parameter $\lambda_{n,k}$ over a quasi-uniform mesh of the mesh size $h=\lambda^{1/4}_{n,k}$;\\
$3^\circ$ Compute $\lambda^{1/2+d/8}_{n,k+1} = \|u_h-y\|_n n^{-1/2}(|u_h|_{2,h}+\|u_h-y\|_n n^{-1/2})^{-1}$.
\end{algorithm}

Now we show several examples to confirm our theoretical analysis. We will always take $\Omega=(0,1)\times(0,1)$ and $\{x_i\}_{i=1}^n$ being uniformly distributed over $\Omega$. We take $u_0=\sin(2\pi x^2+3\pi y) e^{\sqrt{x^3+y}}$, see Figure \ref{exact}. The finite element mesh of $\Om$ is construct by first dividing the domain into $h^{-1}\times h^{-1}$ uniform rectangles and then connecting the lower left and upper right vertices of each rectangle.

\begin{example} \label{numerical.1}
In this example we show that the choice of the smoothing parameter $\lambda_n$ by (\ref{h1}) is optimal. We set $e_i,~ i=1,2,\cdots, n$, being independent normal random variables with variance $\sigma=1$ and $n=2500$. Since $|u_0|_{H^2(\Omega)}\approx 200$, \eqref{h1} suggests the optimal choice of $\lambda_n\approx 3 \times 10^{-6}$. Figure \ref{example.1} shows that $\lambda_n=1 \times 10^{-6}$ is the best choice among 11 deferent choices $\lam_{n}=10^{-k}$, $k=1,2,\cdots,10$. Here we also choose the mesh size $h=\lam^{1/4}_{n}$ according to Theorem \ref{thm:4.2}. 
\end{example}

\begin{figure}[t]
\begin{center}
\includegraphics[width=6cm]{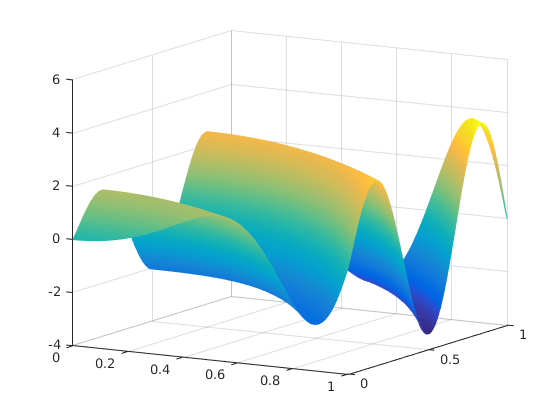} 
\end{center}
\caption{\em The surface plot of the exact solution $u_0$. }
\label{exact}
\end{figure}

\begin{figure}[t]
\begin{center}
\includegraphics[width=7cm]{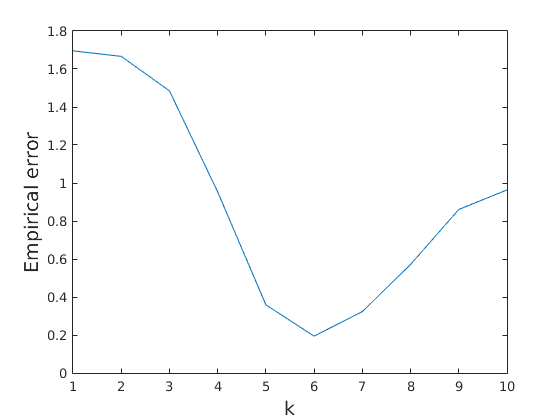}
\end{center}
\caption{\em  The empirical error $\|u_0-u_h\|_n$ for 11 different choices of $\lambda_n=10^{-k}, k=0,1,\cdots,10$. The mesh size $h=\lam_n^{1/4}$.}\label{example.1}
\end{figure}

\begin{example} \label{numerical.2}
In this example we show the empirical error $\|u_0-u_h\|_n$ depends linearly on $\lambda_n^{1/2}$ to confirm \eqref{h2}. We set $e_i,~ i=1,2,\cdots, n$, to be independent normal random variables with variance $\sigma=1$. We take $n$ varying from $2500$ to $9\times 10^4$. In this test we use the optimal $\lambda_n$ and take the mesh size $h=\lambda^{1/4}_n$. Figure \ref{example.2} (a) shows clearly the linear dependence of the empirical error on $\lam_n^{1/2}$. We also run the test for combined random errors, i.e., $e_i=\eta_i+\alpha_i$, where $\eta_i$ and $\alpha_i$ are independent normal random variables with variance $\sigma_1=1$ and $\sigma_2=10$. Figure \ref{example.2} (b) shows also the linear dependence of the empirical error on $\lam_n^{1/2}$. \end{example}

\begin{figure}[t]
\begin{center}
\begin{tabular}{cc}
\includegraphics[width=6cm]{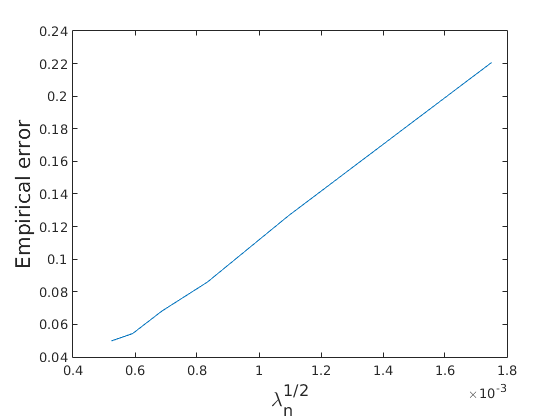} & 
\includegraphics[width=6cm]{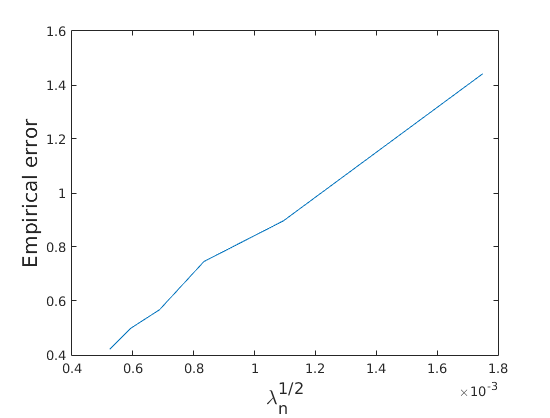} \\
(a) & (b)
\end{tabular}
\end{center}
\caption{\em  (a) The linear dependence of the empirical error $\|u_0-u_h\|_n$ on $\lambda_n^{1/2}$ for $\sigma=1$. (b) The linear dependence of the empirical error $\|u_0-u_h\|_n$ on $\lambda_n^{1/2}$ for combined random noises. }\label{example.2}
\end{figure}

\begin{example} \label{numerical.3}
We test the efficiency of the Algorithm \ref{j1} to estimate the smoothing parameter $\lambda_n$. We will show two experiments of different noise levels. In the first test we set $e_i,~ i=1,2,\cdots, n$, being independent normal random variables with variance $\sigma=1$ and $n=2500$. Figure \ref{example.3} (a) and (b) show clearly that the sequence of $\{\lambda_{n,k}\}$ generated by Algorithm \ref{j1} converges. $\lambda_{n, 16}=4.12\times 10^{-6}$ agrees with the optimal choice $3\times 10^{-6}$ given by \eqref{h1}. Furthermore, $\|u_h-y\|_n=0.99$ provides a good estimate of the variance $\sigma$.   

We now consider the combined random noise. Let $e_i=\eta_i+\alpha_i,~ i=1,2,\cdots, n$, where $\eta_i$ and $\alpha_i$ are independent normal random variables with variance $\sigma_1=1$ and $\sigma_2=10$. It is obvious that $\sigma^2=\E e_i^2=\sigma_1^2+\sigma_2^2=101$. Let $n=4\times 10^4$.
Again Figure \ref{example.3} (c) and (d) show the sequence $\{\lambda_{n,k}\}$ generated by Algorithm \ref{j1} converges. Now $\lambda_{n, 19}=2.16\times 10^{-5}$ which fits well the optimal choice $1.03\times 10^{-5}$ given by \eqref{h1}. Also $\|u_h-y\|_n=10.07$ gives a good estimate of the variance $\sigma$. 
\end{example}

\begin{figure}[t]
\begin{center}
\begin{tabular}{cc}
\includegraphics[width=6cm]{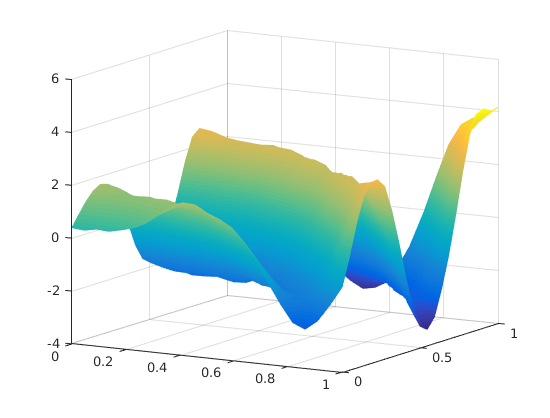} &
\includegraphics[width=6cm]{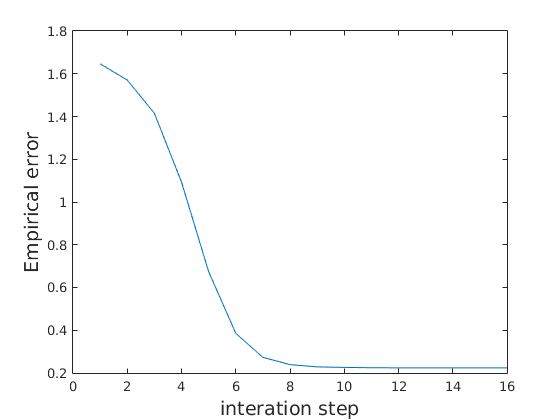} \\
(a) & (b) \\
\includegraphics[width=6cm]{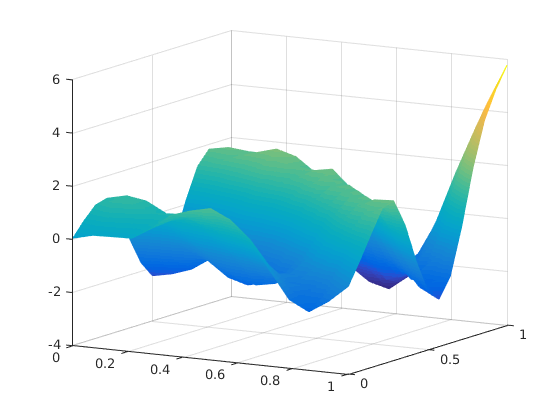} &
\includegraphics[width=6cm]{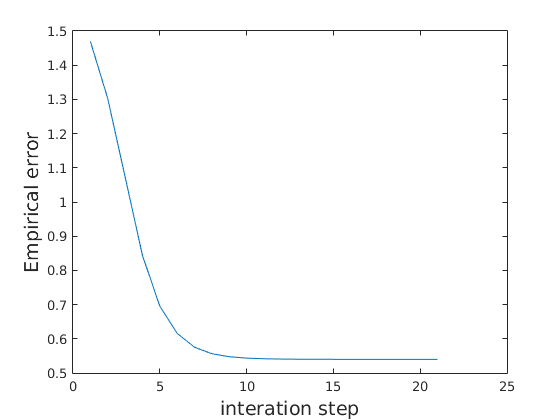} \\
(c) & (d) \\
\end{tabular}
\end{center}
\caption{\em (a) The solution $u_h$ at the end of iteration for $\sigma=1$. (b) The empirical error $\|u_0-u_h\|_n$ of each iteration for $\sigma=1$. (c) The solution $u_h$ at the end of iteration for the combined random error $e_i=\eta_i+\alpha_i$. (d) The empirical error $\|u_0-u_h\|_n$ of each iteration for the combined random error $e_i=\eta_i+\alpha_i$.}
\label{example.3}
\end{figure}

\section{Appendix: Proof of Lemma \ref{lem:3.2} when $d=3$} \label{App:AppendixA}
The proof is very similar to the proof for 2D case in section 3. We will construct $\Pi_hv_h$ by using the three dimensional $C^1$ element of Zhang constructed in \cite{Zhang} which simplifies an earlier construction of Zenisek \cite{Zen}. For any tetrahedron $K\in\cM_h$, the $C^1-P_9$ element in \cite{Zhang} is a triple $(K,P_K,\Lambda_K)$, where $P_K=P_9(K)$ and the set of degrees of freedom $\Lambda_K$ consists of the following $220$ functionals: for any $p\in C^2(K)$, 
\begin{description}
\item[$1^\circ$] The nodal values of $p(a_i), Dp(a_i)(a_j-a_i), D^2p(a_i)(a_j-a_i,a_k-a_i), D^3(a_i)(a_j-a_i,a_k-a_i,a_l-a_i), D^4p(a_i)(a_j-a_i,a_k-a_i,a_l-a_i,a_n-a_i), 1\le i\le 4,1\le j\le k\le l\le n\le 4, i\not\in\{j,k,l,n\}$, where $\{a_i\}^4_{i=1}$ are the vertices of $K$; (120 functionals) 
\item[$2^\circ$] The 2 first order normal derivatives $\pa_{\nu_k}p(a_{ij})$ and 3 second order normal derivatives $\pa^2_{\nu_k\nu_l}p(b_{ij}), \pa^2_{\nu_k\nu_l}p(c_{ij})$ on the edge with vertices $a_i,a_j$, $1\le i\not= j\le 4$, where $\nu_k, k=1,2,$ are unit vectors perpendicular to the edge, and $a_{ij}=(a_i+a_j)/2$, $b_{ij}=(2a_i+a_j)/3$, $c_{ij}=(a_i+2a_j)/3$; (48 functionals) 
\item[$3^\circ$]The nodal value $p(a_{ijk})$ and 6 normal derivatives $\pa_\nu p(a_{ijk}^n)$ on the face with vertices $a_i,a_j,a_k$, $1\le i,j,k\le 4, i\not=j, j\not=k,k\not=i$, $n=1,2,\cdots,6$, where $a_{ijk}$ is the barycenter of the face and $a_{ijk}^1=(2a_i+a_j+a_k)/4,
a_{ijk}^2=(a_i+2a_j+a_k)/4,a_{ijk}^3=(a_i+a_j+2a_k)/4,a_{ijk}^4=(4a_i+a_j+a_k)/6,a_{ijk}^5=(a_i+4a_j+a_k)/6,a_{ijk}^6=(a_i+a_j+4a_k)/6$; (24 functionals) 
\item[$4^\circ$] The nodal values $p(d_i)$, $1\le i\le 4$, at internal points $d_1=(2a_1+a_2+a_3+a_4)/5,d_2=(a_1+2a_2+a_3+a_4)/5,d_3=(a_1+a_2+2a_3+a_4)/5,d_4=(a_1+a_2+a_3+2a_4)/5$. (4 functionals)
\end{description}
Let $X_h$
be the finite element space 
\ben
X_h=\{v_h: v_h|_K\in P_9(K), \forall K\in\cM_h, f(v_h|_{K_1})=f(v_h|_{K_2}),\forall f\in \Lambda_{K_1}\cap\Lambda_{K_2}\}.
\een
It is known that $X_h\subset H^2(\Om)$. We define the operator $\Pi_h$ as follows. For any $v_h\in V_h$, $w_h:=\Pi_h v_h\in X_h$ such that for any $K\in\cM_h$, $w_h|_K\in P_9(K)$, for the degrees of freedom at vertices $a_i$, $1\le i\le 4$,
\beq
\pa^\alpha (w_h|_K)(a_i)=\frac 1{N(a_{i})}\sum_{K'\in\cM_h(a_{i})} \pa^\alpha(v_h|_{K'})(a_{i}),\ \ |\alpha|\le 4,\label{m1}
\eeq
for the degrees of freedom on the edge with vertices $a_i,a_j$, $1\le i\not=j\le 4$,
\be
& &\pa_{\nu_k} (w_h|_K)(a_{ij})=\frac 1{N(a_{ij})}\sum_{K'\in\cM_h(a_{ij})} \pa_{\nu_k}(v_h|_{K'})(a_{ij}), \ \ k=1,2,\label{m2}\\
& &\pa_{\nu_k\nu_l} (w_h|_K)(b_{ij})=\frac 1{N(b_{ij})}\sum_{K'\in\cM_h(b_{ij})} \pa_{\nu_k\nu_l}(v_h|_{K'})(b_{ij}),\ \ k,l=1,2,\label{m3}\\
& &\pa_{\nu_k\nu_l} (w_h|_K)(c_{ij})=\frac 1{N(c_{ij})}\sum_{K'\in\cM_h(c_{ij})} \pa_{\nu_k\nu_l}(v_h|_{K'})(c_{ij}),\ \ k,l=1,2,\label{m4}
\ee
for the degrees of freedom on the faces with vertices $a_i,a_j,a_k$, $1\le i,j,k\le 4,i\not=j, j\not=k, k\not=i$,
\be
& &(w_h|_K)(a_{ijk})=\frac 1{N(a_{ijk})}\sum_{K'\in\cM_h(a_{ijk})}(v_h|_{K'})(a_{ijk}),\label{m5}\\
& &\pa_\nu(w_h|_K)(a_{ijk}^n)=\frac 1{N(a_{ijk}^n)}\sum_{K'\in\cM_h(a_{ijk}^n)}\pa_\nu(v_h|_{K'})(a_{ijk}^n), \ \ n=1,2\,\cdots, 6,\label{m6}\
\ee
and finally for the degrees of freedom at the interior points $d_i, 1\le i\le 4$,
\beq
(w_h|_K)(d_i)=(v_h|_K)(d_i).\label{m7}
\eeq
To show the desired estimate \eqref{b7} in 3D we use the $C^0$-$P_9$ element in \cite{Zhang} which is a triple $(K,P_K,\Theta_K)$, where $P_K=P_9(K)$ and the set of degrees of freedom $\Theta_K$ is defined by replacing some of the degrees of freedom of the $C^1-P_9$ element $\Lam_K$ as follows:
\begin{description}
\item[$1^\circ$] For the edge with vertices $a_i,a_j$, $1\le i\not= j\le 4$, replace the 2 edge first order normal derivatives by $Dp(a_{ij})(a_k-a_{ij}), Dp(a_{ij})(a_l-a_{ij})$ and denote the corresponding nodal basis functions $p_{ij}^k(x),p_{ij}^l(x)$, where $a_k,a_l$ are the other 2 vertices of $K$ other than $a_i,a_j$;
\item[$2^\circ$] For the edge with vertices $a_i,a_j$, $1\le i\not= j\le 4$, replace the 3 edge second order normal derivatives by $D^2p(b_{ij})(a_k-b_{ij},a_l-b_{ij}), D^2p(c_{ij})(a_k-b_{ij},a_l-b _{ij})$
and denote the corresponding nodal basis functions $p_{ij}^{kl}(x),q_{ij}^{kl}(x)$, where $a_k,a_l$ are the other 2 vertices of $K$ other than $a_i,a_j$;
\item[$3^\circ$] For the face with vertices $a_i,a_j,a_k$, $1\le i,j,k\le 4,i\not=j, j\not=k, k\not=i$, replace the face normal derivatives by $Dp(a_{ijk}^n)(a_l-a_{ijk}^n)$ and denote the corresponding nodal basis functions $p_{ijk}^n(x)$, where $a_l$ is the vertex of $K$ other than $a_i,a_j,a_k$, $n=1,2\cdots,6$.
\end{description}

A regular family of this $C^0-P_9$ element is affine-equivalent. For any $v_h\in V_h$, we also define an operator $q_h:=\Lambda_h v_h$ in a similar way as the definition of $\Pi_h$ by replacing the average normal derivatives in \eqref{m2}-\eqref{m4} and \eqref{m6} by the corresponding directional derivatives in the definition of degrees of freedom for the $C^0-P_9$ element. By the same argument as that in the proof of 2D case in section 3 we have
\beq\label{m8}
|v_h-q_h|_{H^m(K)}\le Ch^{2-m}\left(\sum_{K'\in\cM_h(K)}|v_h|_{H^2(K')}^2\right)^{1/2},\ \ m=0,1,2.
\eeq

Next we expend $q_h-w_h\in P_9(K)$ in terms of the nodal basis functions of the $C^0-P_9$ element.  From the definition of the $C^1-P_9$ and $C^0-P_9$ elements, we have $q_h-w_h=\phi_e+\phi_f$ in $K$, where the edge part of the function $q_h-w_h$ is
\ben
\phi_e(x)&=&\sum_{\stackrel{1\le i\not=j\le 4}{\{k,l\}\in \{1,2,3,4\}\backslash\{i,j\},k\not=l}}\Big[D(q_h|_K-w_h|_K)(a_{ij})(a_k-a_{ij})p_{ij}^k(x)\\
& &\hspace{3cm} +\ D(q_h|_K-w_h|_K)(a_{ij})(a_l-a_{ij})p_{ij}^l(x)\Big]\\
& &+\sum_{\stackrel{1\le i\not=j\le 4}{\{k,l\}\in \{1,2,3,4\}\backslash\{i,j\},k\le l}}\Big[D^2(q_h|_K-w_h|_K)(b_{ij})(a_k-b_{ij},a_l-b_{ij})p_{ij}^{kl}(x)\\
& &\hspace{3cm} +\ D^2(q_h|_K-w_h|_K)(c_{ij})(a_k-c_{ij},a_l-c_{ij})q_{ij}^{kl}(x)\Big],
\een
and the face part of the function $q_h-w_h$ is
\ben
\phi_f(x)=\sum_{\stackrel{1\le i,j,k\le 4,i\not= j,j\not= k,k\not=i}{\{l\}\in \{1,2,3,4\}\backslash\{i,j,k\}}}\sum^6_{n=1}D(q_h|_K-w_h|_K)(a_{ijk}^n)(a_l-a_{ijk}^n)p_{ijk}^n(x).
\een
Since the tangential derivatives of $q_h-w_h$ along the edges vanish, we obtain by the same argument as that in the proof of 2D case
in section 3 that
\beq\label{m9}
|\phi_e|_{H^m(K)}\le Ch^{2-m}\left(\sum_{K'\in\cM_h(K)}|v_h|_{H^2(K')}^2\right)^{1/2},\ \ m=0,1,2.
\eeq

On any face $F$ of $K$, $q_h-w_h-\phi_e\in P_9(F)$ and its nodal values at 3 vertices up to 4th order derivatives vanish, its first order normal derivative at the midpoint and two second order normal derivatives at two internal trisection points on 3 edges vanish, and the nodal value at the barycenter also vanishes. This implies $q_h-w_h-\phi_e=0$ on any face of the element $K$.
Let $\tau_{ijk}^n$ be the tangential unit vector on the face of vertices $a_i,a_j,a_k$ such that 
\ben
a_l-a_{ijk}^n=[(a_l-a_{ijk}^n)\cdot\tau_{ijk}^n]\tau^n_{ijk}+[(a_l-a_{ijk}^n)\cdot\nu]\nu.
\een
Now by \eqref{m4}, \eqref{m8}-\eqref{m9}, and the inverse estimate we have
\be
& &|D(q_h|_K-w_h|_K)(a_{ijk}^n)(a_l-a_{ijk}^n)|\nn\\
&\le&|[(a_l-a_{ijk}^n)\cdot \tau_{ijk}^n]D\phi_e(a_{ijk}^n)\tau_{ijk}^n|
+|[(a_l-a_{ijk}^n)\cdot\nu]D(q_h|_K-w_h|_K)(a_{ijk}^n)\nu|\nn\\
&\le&C h^{1/2}\left(\sum_{K'\in\cM_h(K)}|v_h|_{h^2(K')}^2\right)^{1/2}.\label{m10}
\ee
Since a regular family of $C^0-P_9$ element is affine-equivalent, we have $|p_{ijk}^n|_{H^m(K)}\le Ch^{3/2-m}$, $m=0,1,2$.
Therefore, by \eqref{m10} we obtain
\beq\label{m11}
|\phi_f|_{H^m(K)}\le Ch^{2-m}\left(\sum_{K'\in\cM_h(K)}|v_h|_{H^2(K')}^2\right)^{1/2},\ \ m=0,1,2.
\eeq
Combining \eqref{m8}, \eqref{m9}, \eqref{m11} yields the desired estimate \eqref{b7} in 3D since $v_h-w_h=(v_h-q_h)+\phi_e+\phi_f$ in $K$.
The estimate \eqref{b5} can be proved in the same way as the proof for the 2D case in section 3.
This completes the proof. $\Box$

\end{document}